\documentclass{llncs}
  \usepackage{amsmath,amsxtra,amscd,amssymb,latexsym,stmaryrd}
  \usepackage{mathrsfs,dsfont}
\usepackage{graphicx}
\usepackage{color}

\newcommand{\CAT}{\mathop {\mathrm{CAT}}\nolimits} 

\newcommand{\wass}{\mathop {\mathscr{W}_2}\nolimits}

\newcommand{\supp}{\mathop {\mathrm{supp}}\nolimits}

\begin{document}

\title{A geometric study of Wasserstein spaces:\\
 an addendum on the boundary}

\author{J\'er\^ome Bertrand\inst{1} \and Beno\^{\i}t R. Kloeckner\inst{2}
  \thanks{Both authors where partially supported by ANR grant ANR-11-JS01-0011.}}

\institute{Institut de Math\'ematiques, Universit\'e Paul Sabatier (Toulouse) et CNRS 
  \email{bertrand@math.univ-toulouse.fr} \and
       Institut Fourier, Universit\'e Joseph Fourier (Grenoble) et CNRS
   \email{benoit.kloeckner@ujf-grenoble.fr}}

\maketitle

\begin{abstract}
We extend the geometric study of the Wasserstein space $\wass(X)$
of a simply connected, negatively curved metric space $X$
by investigating which pairs
of boundary points can be linked by a geodesic, when $X$ is a tree.
\end{abstract}

Let $X$ be a \emph{Hadamard space}, by which we mean that $X$ is a complete
globally $\CAT(0)$, locally compact metric space. Mainly, $X$ is
a space where triangles are ``thin'': points on the opposite side
to a vertex are closer to the vertex than they would be in the
Euclidean plane. This assumption can also be interpreted as $X$
having non-positive curvature, in a setting more general than manifolds;
it has a lot of consequences (the distance is convex, $X$ is contractible,
it admits a natural boundary and an associated compactification, \dots)
An important example of Hadamard space, on which we shall focus in this paper,
is simply an infinite tree.

The set of Borel probability measures of $X$ having finite second moment
can be endowed with a natural distance defined using optimal transportation,
giving birth to the Wasserstein space $\wass(X)$. It is well-known that
$\wass(X)$ does not have non-positive curvature even when $X$ is a tree.

This note is an addendum to \cite{Bertrand-Kloeckner}, where we defined
and studied the boundary of $\wass(X)$. We refer to
that article and references therein for the background both on Hadamard space
and optimal transportation, as well as for notations. Note that a previous
(long) version of \cite{Bertrand-Kloeckner} contained the present content,
but has been split after remarks of a referee.

Let us quickly sum up the content of \cite{Bertrand-Kloeckner}. The boundary of
$X$ can be defined by looking at geodesic rays, and identifying rays that stay
at bounded distance one to another (``asymptote'' relation). We showed that
there is a natural
boundary $\partial \wass(X)$ of the Wasserstein space that is both close 
to the traditional boundary of Hadamard spaces 
(a boundary point can be defined as an asymptote class of rays)
and relevant to optimal transportation (a boundary point can be seen as a measure
on the cone over $\partial X$, encoding the asymptotic direction and speed distribution of the
mass along a ray). This boundary can be given a topology consistent with both points
of view, and an angular metric; unsurprisingly, it carries geometric information
about $\wass(X)$.

Here we adress the \emph{visibility}, or lack thereof, of $\wass(X)$. A Hadamard
space satisfies the visibility condition if any pair of boundary points can be linked 
by a geodesic (e.g. all trees have the visibility property),
and the same definition makes sense for its Wasserstein space.
It is easily seen that even when $X$ has the visibility condition, 
$\wass(X)$ does not; our result is a complete characterization of
pairs of asymptotic measures that are the ends of a complete geodesic
when $X$ is a tree (Theorem \ref{theo:endsrealizability} in Section \ref{sec:theo}).

Our motivation is twofold: first this result shows how much more constrained
complete geodesics of $\wass(X)$ are compared to complete rays; second
the method of proof involves cyclical monotonicity in an interesting way,
because we have to deal with an optimal transport problem that needs not have a finite
infimum.

\section{A first necessary condition: antipodality}

A complete geodesic $(\mu_t)$ in $\wass(X)$ defines two rays 
and one therefore gets two asymptotic measures, denoted by $\mu_{-\infty}$ and $\mu_{+\infty}$,
also called the \emph{ends} of the geodesic. We recall that these measures are probability
measures on the cone $c\partial X$ over the geodesic boundary of $X$. But
by Proposition 5.2 of \cite{Bertrand-Kloeckner},
these measures are in fact concentrated on $\partial X$, viewed as a subset of $c\partial X$.
In particular, $\wass(X)$ is already far from satisfying the visibility condition.

Note that we shall need to consider measures $\mu$ on the set of unit complete geodesics
$\mathscr{G}^{\mathbb{R}}_1(X)$ that
satisfy the cyclical monotonicity, but such that ${e_t}_\#\mu$ need not have finite second moment.
We still call such maps dynamical transport plan and we say that ${e_{\pm\infty}}_\#\mu$
are its ends. Such a measure $\mu$ defines a complete unit geodesic in $\wass(X)$ if and only if
${e_t}_\#\mu\in\wass(X)$ for some, hence all $t\in\mathbb{R}$. In this section,
we only consider \emph{unit} geodesics even if it is not stated explicitly.

The asymptotic formula (Theorem 4.2 of \cite{Bertrand-Kloeckner})
gives us a first necessary condition valid for any Hadamard space.
Let us say that two points $\zeta,\xi\in \partial X$ are \emph{antipodal} if they are linked by a 
geodesic, that two sets $A_-,A_+\subset\partial X$ are antipodal
if all pairs $(\zeta,\xi)\in A_-\times A_+$ are antipodal, and that
two measures $\nu_-,\nu_+$ on $\partial X$ are antipodal when they are concentrated on 
antipodal sets. Morever, let us call \emph{uniformly antipodal} a pair of measures
whose supports are antipodal.

Given a complete unit geodesic $\mu$, the asymptotic formula readily implies 
that \emph{the ends of any complete unit geodesic of $\wass(X)$ must be antipodal}.

When $X$ is a tree, every pair of boundary points is antipodal
and this condition simply reads that the ends must be concentrated on disjoint sets.

\section{Flows and antagonism}

From now on, $X$ is assumed to be a tree, described as a graph by a couple
$(V,E)$ where: $V$ is the set of vertices;
$E$ is the set of edges, each endowed with one or two endpoints
in $V$ and a positive length. Since $X$ is assumed to be complete, the edges
with only one endpoint are exactly those that have infinite length.
It is assumed that vertices are incident to $1$ or at least $3$ edges, so that
the combinatorial description of $X$ is uniquely determined by its metric structure.
Since $X$ is locally compact, as a graph it is then locally finite.
We fix a base point $x_0\in X$ and use $d$ to denote the distance on $X$.

We say that two geodesics
are \emph{antagonist} if there are two distinct points $x,y$ such that one of the geodesics
goes through $x$ and $y$ in this order, and the other goes through the same points in the other order.


We add to each infinite end a formal endpoint at infinity to unify notations.
Each edge $e$ has two orientations $(xy)$ and $(yx)$ where $x,y$ are its endpoints.
The complement in $\bar X$ of the interior of an edge $e$ of endpoints $x,y$ has two components
$C_x(xy)\ni x$ and $C_y(xy)\ni y$. An oriented edge $(xy)$ has a
\emph{future} $(xy)_+:= C_y(xy)\cap\partial X $ and a \emph{past} $(xy)_-:= C_x(xy)\cap\partial X$.

Assume $\nu_-$ and $\nu_+$ are antipodal measures on $\partial X$.
Define a signed measure by $\nu=\nu_+-\nu_-$ and note that $\nu(\partial X)=0$.
The \emph{flow} (defined by $(\nu_-,\nu_+)$)
through an oriented edge $(xy)$ is $\phi(xy):=\nu((xy)_+)$. 
The flow gives a natural orientation of edges: an oriented edge is 
\emph{positive} if its flow is positive, \emph{neutral} if
its flow is zero, and \emph{negative} otherwise.

Given a vertex $x$, let $y_1,\ldots, y_k$ be the neighbors of $x$ such that
$(xy_i)$ is positive, and $z_1,\ldots,z_l$ be the neighbors of $x$ such that 
$(xz_j)$ is negative. Then $\sum_i \phi(xy_i)=\sum_j \phi(z_jx)$ is called
the flow through $x$ and is denoted by $\phi(x)$. If $x\neq x_0$,
then there is a unique edge starting at $x$ along which the distance
to $x_0$ is decreasing. If this edge is a positive one, $(xy_{i_0})$ say,
then define the \emph{specific flow} through $x$ as 
$\phi^0(x)=\sum_{i\neq i_0} \phi(xy_i)$. If this edge is a negative one,
$(xz_{j_0})$, then let $\phi^0(x)=\sum_{j\neq j_0} \phi(z_jx)$. If this edge
is neutral or if $x=x_0$, then let $\phi^0(x)=\phi(x)$.
Note that 
$\phi(xy)=-\nu((xy)_-)=-\phi(yx)$.

Given
a dynamical transport plan $\mu$, we denote by $\mu(xy)$ the $\mu$-measure
of the set of geodesics that go through an edge $(xy)$ in this orientation,
by $\mu(x)$ the $\mu$-measure of the set of geodesics that pass at $x$,
and by $\mu^0(x)$ the $\mu$-measure of those that are moreover closest to
$x_0$ at this time.
\begin{lemma}\label{lemm:flows}
If $\mu$ is any dynamical transport plan with ends $\nu_\pm$,
then:
\begin{enumerate}
\item for all edge $(xy)$ we have $\mu(xy)\geqslant \max(\phi(xy),0)$,
\item for all vertex $x$ we have $\mu(x)\geqslant\phi(x)$.
\end{enumerate}
and each of these inequality is an equality for all $(xy)$, respectively all
$x$, if and only if $\mu$ contains no pair of antagonist geodesics in its support.
In this case, we moreover have $\mu^0(x)=\phi^0(x)$ for all $x$.
\end{lemma}


\begin{proof}
We prove the first point, the other ones are similar.
Denote by $\mu(C_y(xy))$  the measure of the set of geodesic
that lie entirely in $C_y(xy)$. We have
\[\mu(xy)+\mu(C_y(xy))=\nu_+((xy)_+)=\phi(xy)+\nu_-((xy)_+)\]
and
\[\nu_-((xy)_+)=\mu(C_y(xy))+\mu(yx).\]
It follows that $\phi(xy)=\mu(xy)-\mu(yx)$ so that $\mu(xy)\geqslant\phi(xy)$.
Moreover the case of equality $\mu(xy)=\max(\phi(xy),0)$ implies that $\mu(yx)=0$
whenever $\mu(xy)>0$, and we get the conclusion.
\end{proof}


\begin{lemma}\label{lemm:antagonism}
A dynamical transport plan $\mu$ is $d^2$-cyclically monotone if and only if
$\mu\otimes\mu$-almost no pairs of geodesics are antagonist.
\end{lemma}

\begin{proof}
Assume that the support of $\mu$
contains two antagonist geodesics $\gamma,\beta$ and let $x,y$ be points
such that $\gamma_t=x,\gamma_u=y$ where $u>t$ and $\beta_v=y,\beta_w=x$
where $w>v$. Let $r=\min(t,v)$ and $s=\max(u,w)$. Then 
\[d(\gamma_r,\beta_s)^2+d(\gamma_s,\beta_r)^2<d(\gamma_r,\gamma_s)^2+d(\beta_r,\beta_s)^2)\]
so that the transport plan $(e_r,e_s)_\#\mu$ between $\mu_r$ and $\mu_s$ would not
be cyclically monotone (see Figure \ref{fig:nonoptimal}).

 \begin{figure}[htp]
 \begin{minipage}[c]{.5\linewidth}
    \input{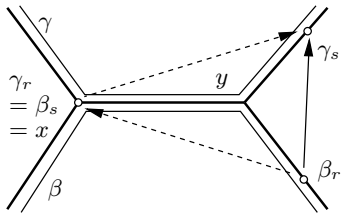}
 \end{minipage}
 \begin{minipage}[l]{.45\linewidth} 
    \caption{The transport plan
    corresponding to the solid arrow is cheaper than the one corresponding to the dashed 
    arrows.}\label{fig:nonoptimal}
 \end{minipage}
 \end{figure}

Assume now that $\mu\otimes\mu$-almost no pairs of geodesics are antagonist.
Let $\tau: X\to \mathbb{R}$ be a function that is continuous, increasing isometric on each
positive edge and constant on neutral edges. Such a function can be defined locally around
any point, and we can design it globally since $X$ has no cycle. By Lemma \ref{lemm:flows},
we see that $\tau$ is isometric when restricted to any geodesic in the support of $\mu$
(such a geodesic must go through positive edges only).
Given times $r<s$, the only $|\cdot|^2$-cyclically monotone transport plan in $\mathbb{R}$
from $\tau_\#\mu_r$ to $\tau_\#\mu_s$ is known to be the increasing rearrangement by convexity of
the cost. Here $\tau_\#\mu_s$ is the $r-s$ translate of $\tau_\#\mu_r$, so that this transport
plan has cost $(r-s)^2$. But $\tau$ is $1$-Lipschitz, so that any transport plan from
$\mu_r$ to $\mu_s$ has cost at least $(r-s)^2$. This proves the cyclical monotonicity of
$\mu$.
\end{proof}

Note that here for example, $\mu_r$ and $\mu_s$ need not have finite second moment; however
$\mu$ induces a transport plan with finite cost between them, and that peculiarity has
therefore no incidence on the proof.

\section{Gromov product}

Before we state the main result, let us turn to a second point of view.

Given $\xi_-\neq\xi_+$ in $\partial X$, we denote by 
$(\xi_-,\xi_+)\subset X$ the locus of a geodesic whose ends are $\xi_-$ and $\xi_+$.
Then we write $D_0 (\xi_-,\xi_+)$
the distance between the base point $x_0$ and the geodesic $(\xi_-,\xi_+)$.
Since $X$ is a tree, this quantity is
equal to what is usually called the \emph{Gromov product} $(\xi_-\cdot\xi_+)_{x_0}$, 
see e.g. \cite{BH};
however the present definition is adapted to our needs.  Set $D_0(\xi,\xi)=\infty$
and denote by $\gamma(\xi_-,\xi_+)$ the parametrized unit complete geodesic whose ends
are $\xi_\pm$ at $\pm\infty$, and such that its time $0$ realizes $D_0$:
$d(x_0,\gamma(\xi_-,\xi_+)_0)=D_0(\xi_-,\xi_+)$.

For any $\varepsilon>0$,
$e^{-\varepsilon D_0}$ metrizes the cone topology on
$\partial X$ so that
$D_0(\xi_n,\zeta)\to\infty$ if and only if $\xi_n\to\zeta$.
Moreover by compactness if $D_0(\xi_n,\zeta_n)\to\infty$ then
there are increasing indices $(n_k)$ such that $\xi_{n_k}$ and $\zeta_{n_k}$ converge to 
a common point. $D_0$ is continuous, and locally constant outside the diagonal; the map
\begin{eqnarray*}
F : \partial X\times \partial X &\to& \mathscr{G}^{\mathbb{R}}_1(X) \\
(\xi_-,\xi_+) &\mapsto& \gamma(\xi_-,\xi_+)
\end{eqnarray*}
is easily seen to be continuous.
when $\mathscr{G}^{\mathbb{R}}_1(X)$ is endowed
with the topology of uniform convergence on compact subsets. 

Since $F$ is a right inverse to $(e_{-\infty},e_{+\infty})$,
a transport plan $\Pi\in\Gamma(\nu_-,\nu_+)$ can always be written
$(e_{-\infty},e_{+\infty})_\#\mu$ by taking $\mu=F_\#\Pi$.
We shall denote also by $D_0$ the map $\gamma\mapsto d(x_0,\gamma)$ where $\gamma$
is any parametrized or unparametrized complete geodesic.

\begin{lemma}\label{lemm:antagonism2}
A transport plan $\Pi_0\in\Gamma(\nu_-,\nu_+)$ such that 
\[\int -D_0^2\,\Pi_0>-\infty\]
is $-D_0^2$-cyclically monotone
if and only if $F_\#\Pi_0$ contains no pair of antagonist geodesics in its
support. In this case, $\Pi_0$ is a solution to the optimal transport problem 
\begin{equation}
\inf_{\Pi \in \Gamma(\nu^-,\nu^+)} \int -D_0^2(\xi,\zeta) \,\Pi(d\xi,d\zeta).
\label{eq:D02problem}
\end{equation}
\end{lemma}

\begin{proof}
Consider $\Pi_0\in\Gamma(\nu_-,\nu_+)$ and let $\mu=F_\#\Pi_0$.

Assume first that there are antagonist geodesics $\gamma,\beta$ in the support
of $\mu$. Then permuting $\gamma_{+\infty}$ and $\beta_{+\infty}$
contradicts the $-D_0^2$-cyclical monotonicity.

To prove the other implication, assume that $\supp\mu$ contains no pair
of antagonist geodesics but that $\Pi_0$ does not achieve the infimum
in \eqref{eq:D02problem}. This happens notably when $\Pi_0$ is not cyclically monotone.

Then there is some $\Pi_1\in\Gamma(\nu_-,\nu_+)$ such that 
$\int -D_0^2\,\Pi_0>\int -D_0^2\,\Pi_1$.
If $F_\#\Pi_1$ has couples of antagonist geodesics in its support,
then we can still improve $\Pi_1$. Choosing any numbering $e_2,e_3,\ldots$ of
the non-oriented edges of $X$, we inductively construct transport plans 
$\Pi_2,\Pi_3,\ldots$ in $\Gamma(\nu_-,\nu_+)$ such that $F_\#\Pi_k$ has no antagonist
geodesics through the edges $e_2,\ldots,e_k$ in its support, and 
$(-D_0^2)_\#\Pi_k([x,+\infty)) \leq (-D_0^2)_\#\Pi_{k-1}([x,+\infty))$ for all $x$ 
(proceed as follows: for all
$(\zeta,\xi)$ going through $e_k$ in the negative direction, replace $\xi$
by some $\xi'$ in the future of $e_k$ and corresponding to a $(\zeta',\xi')\in\supp\Pi_{k-1}$,
and replace $(\zeta',\xi')$ by $(\zeta',\xi)$; there are many choices to do but they can
be made in a arbitrary manner). Then we get
$\int -D_0^2\,\Pi_{k-1}\geqslant\int -D_0^2\,\Pi_k$
where some, or even all of this integrals can be negative infinite.

We shall use a weak convergence, but $-D_0^2$ is not bounded; we therefore introduce
the functions $f_T=-\min(D_0^2,T)$ for $T\in\mathbb{N}$. For all $T$,
the transport plans
$\Pi_k$ also satisfy 
\[\int f_T\,\Pi_{k-1}\geqslant\int f_T\,\Pi_k\]

Since $\partial X$ is compact, so is $\Gamma(\nu_-,\nu_+)$ and we can extract a
subsequence of $(\Pi_k)$ that weakly converges to some $\tilde\Pi$, and $\tilde\mu:=F_\#\tilde\Pi$ has
no pair of antagonist geodesics in its support.
 
The monotone convergence theorem implies that for $T$ large enough,
we have $\int f_T\,\Pi_1<\int -D_0^2\,\Pi_0$, and by weak convergence we get
$\int f_T\,\tilde\Pi\leqslant\int f_T\,\Pi_1$.
Since $-D_0^2\leqslant f_T$, we get
\[\int -D_0^2\,\tilde\Pi<\int -D_0^2\,\Pi.\]
But then by Lemma
\ref{lemm:flows} we get that $\tilde\mu^0(x)=\mu^0(x)$ for all $x\in V$,
and since $\int -D_0^2\,\Pi_0=\sum_x -d^2(x_0,x)\mu^0(x)$ it follows that
$\int -D_0^2\,\Pi_0$ and $\int -D_0^2\,\tilde\Pi$ must be equal, a contradiction.
\end{proof}

Note that the lemma stays true if $-D_0^2$ is replaced with any decreasing function
of $D_0$, but that we shall need precisely $-D_0^2$ later.

\begin{remark}
In this proof we cannot
use Theorem 4.1 of \cite{Villani} since we do not have the suitable lower bounds
on the cost.
\end{remark}

\section{Characterization of ends}\label{sec:theo}

We can now state and prove our result.

\begin{theorem}\label{theo:endsrealizability}
Assume that $X$ is a tree and let $\nu_-$, $\nu_+$ be two antipodal measures
on $\partial X$. The following are equivalent:
\begin{enumerate}
\item there is a complete geodesic in $\wass(X)$ with $\nu_\pm$ as ends;
\item\label{enumi:Gromov} the optimal transport problem \eqref{eq:D02problem} is finite:
\[\inf_{\Pi \in \Gamma(\nu^-,\nu^+)} \int -D_0^2(\xi,\zeta) \,\Pi(d\xi,d\zeta)
  >-\infty;\]
\item the specific flow defined by $\nu_\pm$ satisfies 
  \[\sum_{x\in V} \phi^0(x) d(x,x_0)^2 <+\infty.\]
\end{enumerate}
When these conditions are satisfied, then the optimal transport problem \eqref{eq:D02problem}
has a minimizer
$\Pi_0$ and $\Gamma_\#\Pi_0$ define a geodesic of $\wass(X)$ with the prescribed ends.

Moreover, the above conditions are satisfied as soon as $\nu_\pm$ are uniformly antipodal.
\end{theorem}

\begin{proof}
First assume that there is a complete geodesic in $\wass(X)$ with $\nu_\pm$ as ends
and denote by $\mu$ one of its displacement interpolations;
by Lemma \ref{lemm:antagonism}, the support of $\mu$ does not contain
any pair of antagonist geodesics.
From Lemma \ref{lemm:flows} it follows that 
\[\sum_{x\in V} \phi^0(x) d(x,x_0)^2 \leqslant \int_X d(x,x_0)^2 \mu_0(dx)<\infty\]
since by hypothesis $\mu_0\in\wass(X)$.
We have also
\[\int -D_0^2(\xi,\zeta) \,\Pi_0(d\xi,d\zeta)\geqslant -\int_X d(x,x_0)^2 \mu_0(dx)>-\infty\]
so that Lemma \ref{lemm:antagonism2} implies that
$\Pi_0:=(e_{-\infty},e_{+\infty})_\#\mu$ is a solution to problem \ref{eq:D02problem}.

Now consider the case when $\nu_\pm$ are uniformly antipodal. Since the supports of $\nu^-$
and $\nu^+$ are disjoint, the map $D_0$, when 
restricted to $\supp \nu^- \times  \supp \nu^+$, is bounded. Therefore, since it is a continuous 
map, the optimal mass transport problem is well-posed and admits minimisers.

More generally when the infimum in problem  \eqref{eq:D02problem} is finite, 
by using the regularity of Borel 
probability measures on $\partial X$ we can approximate $\nu^-$ and $\nu^+$ by probability 
measures whose supports are disjoint sets. Then, the previous paragraph gives us a sequence 
of plans which are $-D_0^2$-cyclically monotone. Since $D_0$ is a continuous map, Prokhorov's 
theorem allows us to extract a converging subsequence whose limit $\Pi_0$ is $-D_0^2$-cyclically monotone.
By the finiteness assumption, 
\[\int -D_0^2(\xi,\zeta) \,\Pi_0(d\xi,d\zeta)>-\infty\]
and $\Pi_0$ is a $-D_0^2$-optimal transport plan.

As soon as a minimizer $\Pi_0$ to \eqref{eq:D02problem} exists,by Lemma \ref{lemm:antagonism2}
$\mu:=\Gamma_\#\Pi_0$ is a dynamical transport plan that has
no antagonist pair of geodesics in its support. By Lemma \ref{lemm:antagonism}
$\mu$ is cyclically monotone. By its definition
\[\int_X d(x,x_0)^2 \mu_0(dx)=\int D_0^2(\xi,\zeta) \,\Pi_0(d\xi,d\zeta)<+\infty\]
so that $\mu$ defines a geodesic of $\wass(X)$. It has the prescribed ends
since $\Pi_0\in\Gamma(\nu_-,\nu_+)$.

We have only left to consider the case when $\sum_{x\in V} \phi^0(x) d(x,x_0)^2 <\infty$.
For this, let us construct a suitable complete geodesic by hand.
Let $\tau$ be the time function such that $\tau(x_0)=0$, as in the proof of Lemma
\ref{lemm:antagonism}. 
The levels of the time function are finite unions of isolated points and of subtrees of $X$ all of
whose edges are neutral. Indeed, consider a point $a$: if it lies inside a neutral edge, then
all the edge has time $\tau(a)$. Otherwise, let $(xy)$ be the orientation of this edge that is positive:
points on $[x,a)$ have time lesser than $\tau(a)$, while points on $(a,y]$ have time greater
than $\tau(a)$. If $a$ is a vertex, then similarly one sees that nearby $a$, only the points lying on
an incident neutral edge can have time equal to $\tau(a)$.
Let $\dot\tau^{-1}(t)$ be the union of the isolated points of the level
$\tau^{-1}(t)$ and of the points that lie (in $X$) on the boundary of the neutral subtrees of the same level.
In other words, $\dot\tau^{-1}(t)$ is the level $t$ of the map induced by $\tau$ on the subforest
of $X$ where all neutral edges have been removed.

Define now 
\[\tilde\mu_t=\sum_{a\in\dot\tau^{-1}(t)} \phi(a)\,\delta_a\]
where $\phi(a)=\phi(xy)$ if $a$ lies inside a positive edge $(xy)$.
Note that $\tilde\mu_t$ is a probability measure thanks to the antipodality
of $\nu_-$ and $\nu_+$: without it, it would have mass less than $1$.
It is a good first candidate to be the geodesic we are looking for, except the
second moment of $\mu_t$ need not be finite! To remedy this problem, proceed as follows.
First, there is a displacement interpolation $\tilde\mu$ of $(\tilde\mu_t)$, which is
a probability measure on $\mathscr{G}^{\mathbb{R}}_1(X)$.
 Now, construct a random geodesic $\gamma$ as follows : draw $\tilde\gamma$
with law $\tilde\mu$, and let $\gamma$ be the geodesic that has the same geometric locus
and the same orientation as $\tilde\gamma$, and such that $\gamma$ is nearest to $x_0$
at time $0$. The condition $\sum_{x\in V} \phi^0(x) d(x,x_0)^2 <\infty$ ensures that the law of
$\gamma_0$ has finite second moment and $(\mu_t)$ is the desired geodesic.
\end{proof}

The example shown in Figure \ref{fig:notends} shows that antipodality is not sufficient
for $\nu_\pm$ to be the ends of a geodesic.

\begin{figure}[htp]\begin{center}
\input{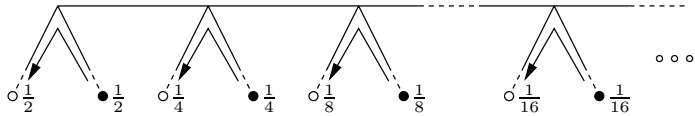}
\caption{The measures $\nu_-$ (black dots) and $\nu_+$ (white dots) are antipodal. However, the only possible geodesics
  having these measures as ends, depicted by simple arrows, are not in
  $\wass(X)$ if the horizontal edges are long enough.}\label{fig:notends}
\end{center}\end{figure}

As a last remark, let us stress that the condition \ref{enumi:Gromov}
in Theorem \ref{theo:endsrealizability} is clearly necessary
when $X$ is a general Hadamard space, but it might not be a
sufficient condition in general. 

\bibliographystyle{smfalpha}
\bibliography{biblio.bib}

\end{document}